\documentclass[smallcondensed,envcountreset]{svjour3}

\usepackage{pifont}
\usepackage{mathptmx}       
\usepackage{helvet}         
\usepackage{courier}        
\usepackage{type1cm}        
\usepackage{makeidx}         
\usepackage{graphicx}        
\usepackage{multicol}        
\usepackage[bottom]{footmisc}
\usepackage{amsfonts}
\usepackage{amssymb}
\usepackage{url}
\usepackage{color}

\usepackage{subfigure} 

\newcommand{\R}{{\sf I\hspace{-.15em}R}}
\newcommand{\N}{{\sf I\hspace{-.15em}N}}
\newcommand{\p}{{\sf I\hspace{-.15em}P}}

\newcommand{\nc}{\newcommand}
\nc{\nv}{\textsl {v}}

\makeindex             

\begin{document}

\title{Semicoercive  Variational Inequalities - From Existence to Numerical Solution of Nonmonotone Contact Problems\thanks{*} }
\titlerunning{Semicoercivity}
\author{Nina Ovcharova \and Joachim Gwinner}
\date {\today}
\institute{N. Ovcharova
            \at Institute of Mathematics, Department of Aerospace Engineering, Universit\"{a}t der Bundeswehr M\"{u}nchen, Germany,  \email{nina.ovcharova@unibw.de} (\ding{41})
       \and J. Gwinner 
            \at Institute of Mathematics, Department of Aerospace Engineering, Universit\"{a}t der Bundeswehr M\"{u}nchen, Germany,  \email{joachim.gwinner@unibw.de}}
 
%
%

\maketitle

\begin{center}
{\it Dedicated to the memory of Professor V.F. Demyanov and 
Professor S. Schaible}  
\end{center}

\abstract{ In this paper we present a novel numerical solution procedure for semicoercive hemivariational inequalities. As a model example we consider a unilateral semicoercive contact problem with nonmonotone friction and provide numerical results
for benchmark tests.}

\keywords{Semicoercivity  \and Pseudomonotone bifunction \and  Hemivariational inequality \and Plus function \and Smoothing approximation  \and  Finite element discretization \and Unilateral contact}
\vspace{0.2cm}
\hspace{-0.6cm} {\bf AMS} 74G15, 74M15,  74S05
\section{Introduction}

This paper presents a contribution to constructive nonsmooth analysis
\cite{DemyanovRubinov}
in view of application to nonsmooth continuum mechanics
\cite{DemyanovStavroulakis}.
It is devoted to the numerical solution of nonmonotone semicoercive contact problems in solid mechanics modelled by hemivariational inequalities, a class of variational inequalities (VIs) introduced and studied by Panagiotopoulos 
\cite{Panagiotopoulos1993}, see also \cite{Naniewicz,Panagiotopoulos1998,Goeleven}.
Here we treat unilateral contact problems with nonmonotone friction
\cite{Baniotopoulos}, which occur with adhesion and delamination in the delicate situation, where the body is not fixed along some boundary part, but is only subjected to surface tractions and body forces. Thus there is a loss of coercivity leading to so-called semicoercive or noncoercive variational problems
 \cite{AttouchButtazzo,Penot}.

The existence theory of hemivariational inequalities (HVIs) and of more abstract 
topologically (in the sense of Br\'{e}zis \cite{Brezis}) pseudomonotone 
 VIs, 
also in the semicoercive case is well documented in the literature.
Without claim of completeness we can cite 
[11-28]  in chronological order 

While there are studies of numerical solution methods for {\it coercive} HVIs,
 see the book \cite{Haslinger} of Haslinger, Miettinen, and  Panagiotopoulos
and also the more recent paper \cite{Hinter}, and there are works on the numerical treatment of semicoercive {\it monotone} VIs, 
see the book of Kaplan and Tichatschke \cite{Kaplan} 
and the papers \cite{Gwi-91,Gwi-94,Spann,Dostal_98,AdlyGoel,Dostal_03,Namm},  
to the best of the authors' knowledge,  
efficient methods for the numerical solution of semicoercive HVIs supported by rigorous mathematical analysis are missing. It is the purpose of the present paper to initiate  
work in this direction. To this end we extend the  convergence result in \cite{Gwi_Ov2015} 
 for a general approximation scheme for the solution of pseudomonotone VIs
 to the semicoercive case. Then based on this result, we combine finite element discretization with regularization techniques of nondifferentiable optimization
\cite{Ovcharova,Ov_Gwin2014} to arrive at a novel solution procedure for semicoercive HVIs.

The paper is organized as follows. In the next section we revisit the existence theory for pseudomonotone VIs that is is necessary for the subsequent development. Here we exhibit also an interesting relation between a sign condition for the data of the pseudomonotone VI,
that is crucial for existence, and the more recent notion of well-positioned convex sets 
in the stability theory for semicoercive pseudomonotone VIs due to 
Adly, Th{\'e}ra, and Ernst \cite{Adly-02,Adly-04}. Then we adapt the general approximation scheme of Glowinski \cite{Glowinski} to pseudomonotone VIs and provide convergence results
thus extending \cite[Theorem 3.1]{Gwi_Ov2015}  to the semicoercive situation. 
Then we apply this convergence theorem to HVIs in linear elasticity, combine finite element discretization with regularization techniques of nondifferentiable optimization from  
\cite{Ovcharova,Ov_Gwin2014} to obtain a novel solution procedure for semicoercive HVIs.
To demonstrate its efficiency we finally provide numerical results for benchmark tests.

\label{sec:1}
\section{ Semicoercive pseudomonotone variational inequality}\label{sec:2} \label{section11}
In this section we revisit the existence theory for pseudomonotone VIs that is necessary for the subsequent treatment of HVIs. Here we follow 
\cite{Gwinner_PhD,Jeggle,Gwi-94,GoelGwin}
for the setting of pseudomonotone bifunctions and semicoercivity.
 Moreover, we also exhibit a relation to the more recent notion of well-positioned convex sets introduced and studied in \cite{Adly-02,Adly-04}. 
 
Let $V$ be a reflexive Banach space and $V^*$ its dual. We denote by $\langle \, \cdot \,,\cdot \,\rangle$  the duality pairing between $V$ and $V^*$, and by $\| \cdot\|$ and $\| \cdot\|_*$ the norm and the dual norm on $V$ and $V^*$, respectively. We are concerned with a linear semicoercive operator $A$ from $V$ to $V^*$, i.e. there exists some positive constant $c_0$ such that
$$
\langle Av, v\rangle \geq c_0 |v|^2 \quad \forall v \in V, 
$$ 
where $|\cdot|$ denotes a lower semicontinuous seminorm on $V$. Thus, the kernel $Y$ of $|\cdot|$, defined by 
$$
Y=\{ y \in V \,:|y|=0\, \}
$$
is a closed nontrivial subspace.

Further, let $g\in V^*$ be a continuous linear form, $K\subseteq V$ a nonvoid closed convex set such that $0\in K$, and consider
$\varphi : K\times K \to \R$ such that $\varphi (\cdot \,, \,\cdot)$ is pseudomonotone, $\varphi(\cdot, v)$ is upper semicontinuous on the intersection of $K$ with any finite dimensional subspace of $V$
and, moreover, there exists some positive constant $c$ such that
\begin{equation} \label{ass1}
 \varphi(v,0) \leq c \|v\| \quad \forall v \in V. 
\end{equation}
We recall that the bifunction $\varphi (\cdot \, , \,\cdot) $ is pseudomonotone if $u_n\rightharpoonup u$ (weakly) in $V$ and \\ $\displaystyle \liminf _{n \to   \infty}  \varphi (u_n,u) \geq 0$ implies that $\displaystyle \limsup _{n \to \infty} \varphi (u_n,v)\leq \varphi (u,v)$ for all $v\in V$. 

In this paper we deal with the semicoercive variational problem  $ VI(A,\varphi, g, K):$ Find $u\in K$ such that 
\[
\langle A u, v-u \rangle + \varphi(u,v) \geq \langle g, v-u \rangle \quad \forall v\in K.
 \]
From \cite{Gwinner_PhD,Gwi-94} we adopt the following assumption:
\begin{description}
\item [(A1$^s$)]  for any sequence $\{v_n\}$ with $|v_n|\to 0$, $v_n\rightharpoonup
  v$ and $\|v_n\|\geq \eta$ for some $\eta>0$ there exists a subsequence
  $\{v_{n_k}\}_{k\in \N}$ such that $v_{n_k}\to v$ in $V$.
\end{description}
According to \cite{Gwinner_PhD}, \cite[Section 5.3]{Jeggle} 
this condition is fulfilled if
 the norm $\|\cdot\|$ on $V$  is equivalent to $\||\cdot\||+ |\cdot|$, where
  $\||\cdot\||$ is another norm on $V$, the dimension of $Y$ is finite and there   exists $\alpha >0$ such that 
$$
\displaystyle \inf _{y\in Y} \||v-y\|| \leq \alpha |v| \quad \forall v\in V,
$$
from which the seminorm $|\cdot|$ is continuous. 

In addition, we resume from \cite{Gwinner_PhD,Gwi-94} that  
\begin{description}
\item [(A2$^s$)] either $(i)~ Y\cap K$ is bounded or 
$(ii)~ g$ satisfies  
\[
\langle g, y\rangle  <  0  \quad \forall y\in \{Y\cap K\} \setminus \{0\}  \,.
\]
\end{description}

Condition (A2$^s$) $(ii)$ 
implies that the linear functional $g \in V^*$ is bounded from above on $K\cap Y$. In other words,  $g$ belongs to the barrier cone of the set
$Y\cap K$ defined by
\[
B(Y\cap K) =\{g\in V^* \, : \, \sup_{v\in Y\cap K} \, \langle g, v \rangle < \infty \}.
\]
Let the interior of the barrier cone be nonempty. Then we can show that ~$g\in \, \mbox{int} \, B(Y\cap K)$.  According to the terminology used in \cite{Adly-02,Adly-04}, this is equivalent to the property 
that the set $Y\cap K$ is well-positioned. 
\begin{proposition} Let $\mbox{int}\, B(Y\cap K)\neq \emptyset$. 
Suppose  (A1$^s$).
Then condition (A2$^s$) $(ii)$ implies that
$g \in \, \mbox{int}\, B(Y\cap K)$. 
\end{proposition}
\begin{proof} Suppose by contradiction that $g$ does not belong to $ \mbox{int}\, B(Y\cap K)$. Then, for any sequence  $\{\varepsilon_n\}$ with $\varepsilon_n\to 0^+$, there exist  $\|h_n\|_{*}\leq 1$ and $y_n \in Y\cap K$ such that 
\begin{equation} \label{eq1}
\langle g + \varepsilon_n h_n\,, \,y_n \rangle >n.
\end{equation}
If $y_n=0$, we trivially get a contradiction. Therefore, let $y_n \neq 0$. Using (A2$^s$) $(ii)$, it follows that
\[
\varepsilon_n \|y_n\| \geq
\varepsilon_n\|h_n\|_*\|y_n\| >
 \langle g + \varepsilon_n h_n \, , \, y_n \rangle > n,
\]
and therefore,  $\displaystyle \lim_{n\to \infty}\|y_n\|=+ \infty$. Define  $t_n:=\frac{1}{\|y_n\|}$ and consider the sequence $\{t_n y_n \}$. Since $0\in K$ and $K$ is convex it follows that $t_ny_n \in K$ for almost all $n$. 
Moreover, $t_ny_n \in Y$.
Since $\|t_ny_n\|=1$,  we can extract a subsequence that converges weakly to some $\tilde{y}\in V$. In virtue of (A1$^s$) we can pass to a strongly
convergent subsequence. Hence $\tilde{y} \not= 0$. Since  $Y\cap K$ is closed, we can conclude that $\tilde{y}\in Y\cap K \setminus \{0\}  $. 
 
   Multiplying (\ref{eq1}) by $t_n$, passing to the limit as $n\to \infty$ and taking into account  (A2$^s$) $(ii)$ again, we obtain
 \[
 0> \langle g, \tilde{y} \rangle  \geq \displaystyle \liminf_{n\to \infty} n \, t_n \geq 0,
  \]  
  which yields a contradiction. \qed
  \end{proof}
If $K$ is a nonvoid convex cone
 - note that this is the relevant case in view of our subsequent application to HVIs -,
the converse is also true. 
\begin{proposition} Let $K$ be a nonvoid convex cone. Then 
\[
g\in  \, \mbox{int}\, B(Y\cap K) \Rightarrow \langle g, v\rangle <0 \quad \forall y\in \{Y\cap K\} \setminus \{0\}.
\]
\end{proposition}
\begin{proof} First, since $K$ is cone, the barrier cone coincides with the polar cone, i.e.
\begin{equation}\label{pr2}
B(Y\cap K)= (Y\cap K)^{-}.
\end{equation}
Because $g\in  \, \mbox{int} (Y\cap K)^{-}$, 
there exists $\varepsilon >0$ such that for all $h\in V^*$ with 
$\|h\|_* \leq 1$ we have $g+\varepsilon h \in (Y\cap K)^{-}$. 

Now assume that there exists $\tilde{y}\in (Y\cap K) \setminus \{0\}$ such that $\langle g, \tilde{y} \rangle \geq 0.$ Then  by an easy consequence of the Hahn-Banach extension theorem there exists $\tilde{h} \in V^*$ such that $\|\tilde{h}\|_*=1$ and $\langle \tilde{h}, \tilde{y} \rangle =\|\tilde{y}\|$. Then
\[
\langle g+ \varepsilon \tilde{h}, \tilde{y} \rangle 
\geq  \varepsilon \langle \tilde{h}, \tilde{y} \rangle 
= \varepsilon \|\tilde{y}\| >0
\]
contradicting that $g + \varepsilon \tilde{h} \in (Y \cap K)^{-}$.
\qed
\end{proof}

\begin{remark}
The barrier cone in finite-dimensional setting has a nonempty interior. 
While this is trivial for a compact $Y \cap K$,
let us show this in the relevant case of a convex closed cone $K$. 
Then indeed,  we have $Y\cap K =\displaystyle \sum _{k=1}^n \R_{-} w_k$ with $\|w_k\|=1$. Construct $h_k\in V^*$ such that
\[
\|h_k\|_{*}=1 \quad \mbox{and} \quad \langle h_k, w_l\rangle = \delta_{kl}. 
\]
Let  $h\in V^*$ be such that $\|h\|_*\leq \frac{1}{2}$. Hence. $h=\displaystyle \sum _{k=1}^n \gamma_k h_k$ and $\langle h, w_k\rangle = \gamma_k$. Moreover, we have
\[
\frac{1}{2}\geq \|h\|_* =  \displaystyle \sup_{\|v\|=1} |\langle h, v\rangle | \geq |\langle h, w_k\rangle |=|\gamma_k|.
\]
With $e:=\displaystyle \sum _{k=1}^n h_k$ we have 
$
e+h = \displaystyle \sum _{k=1}^n (1+\gamma_k) h_k \in (Y\cap K)^-,
$
which due to (\ref{pr2}) implies the nonemptiness of the interior of the barrier cone.
\end{remark}

Let us now turn back to  the problem  $VI(A,\varphi, g, K)$.  We provide 
the subsequent existence theorem in view of application to HVIs. Since in HVIs, the bifunction $\varphi$ (see the later definition (\ref{bifunction})), cannot be controlled by the seminorm, but only by the norm (see (\ref{ass1}) that is proved in \cite[Lemma 15]{Ovcharova} for HVIs), we have to strengthen the condition (A2$^s$) (ii) to
\begin{description}
\item [(iii)] 
$\exists c>0 \, :\, \langle g, y\rangle < -c \; \mbox{for all} \; y \in Y\cap K \; \mbox{with} \; \|y\|=1. $
 \end{description}
Then we can show the following existence result without any compact imbedding assumption.  
\begin{theorem} \label{th:1}
Let $A :V\to V^*$ be a semicoercive  linear operator
and $\varphi :V\times V \to \R$ be a pseudomonotone bifunction 
such that (\ref{ass1}) holds and
such that $\varphi(\cdot, v)$ is upper semicontinuous on the intersection of $K$ with any finite dimensional subspace of $V$. Under assumptions (A1$^s$) and 
(A2$^s$) $(i)$ or $(iii)$, the problem $ VI(A,\varphi, g, K)$ has a solution. 
\end{theorem}
\begin{proof} For the proof we use a recession argument that goes back to
Stampacchia \cite{Stamp}, Hess \cite{Hess} and Schatzman \cite{Schatzman}. Similar reasoning based on a recession analysis can be found in 
\cite{Gwinner_PhD,Gwi-91,Gwi-94,GoelThera95,AdlyGoelThera96,Gwi-97,Chadli-98,Chadli-99,Liu-03,Chadli-07}.
 Another approach to existence results for semicoercive pseudomonotone variational inequalities is a regularization procedure based on adding of a coercive term, see
 \cite{Gwi-97,Nani-00,Chadli-04,Chadli-14}.

In view of  \cite[Theorem 3.9]{Gwinner_PhD}, \cite[Proposition]{Hess} 
it is sufficient to show the existence of a constant $R>0$ such that
\[
\langle A v, -v \rangle + \varphi(v,0) + \langle g, v\rangle < 0 \quad \forall v \in K \quad \mbox{with} \quad \|v\|=R.
\]
Assume the contrary, i.e., there exists a sequence $v_n\in K $ such that $\|v_n\|\to \infty$ and
\[
\langle A v_n, -v_n \rangle + \varphi(v_n,0) + \langle g, v_n\rangle \geq 0.
\]
Hence, using the semicoercivity of the operator $A$ and (\ref{ass1}), it follows that
\begin{equation} \label{semieq1}
c_0|v_n|^2 \leq \langle A v_n, v_n \rangle \leq \langle g, v_n \rangle + \varphi (v_n, 0) \leq \|g\|_*\|v_n\| + c \|v_n\|. 
\end{equation}  
Set $y_n=\frac{v_n}{\|v_n\|}$. Since $0\in K$ and $K$ is convex,  it follows that
 $y_n \in K$ for large enough $n$ as well. Moreover, $\|y_n\|=1$ and consequently, we  can extract a subsequence that converges weakly to some $y$ in the weakly closed set $K$. 
 
Now, we show that $|y_{n}|_V\to 0$. Assume not. Consequently, there exists a 
subsequence, again denoted by $\{y_{n}\}$, such that $|y_{n}|\geq c_3>0$.
Dividing  (\ref{semieq1}) by $\|v_{n}\|$ yields
$$
c_0 \|v_{n}\| \, |y_{n}|^2 \leq \|g\|_{*} + c
$$
and hence,
$$
|y_{n}| \leq \frac{\|g\|_{*}+c}{c_0|y_{n}|\|v_{n}\|}\leq \frac{\|g\|_{*} + c} {c_0\, c_3 
\|v_{n}\|}.
$$
Since $\|v_{n}\|\to \infty$, we arrive at a contradiction, and therefore $|y_{n}|\to 0$. 
Further, by {\bf (A1$^s$)}, we can extract a subsequence again denoted by $\{y_n\}$ that converges strongly  
to $y$ in $V$. Since $ \|\cdot\| $ is   continuous, we have  $\|y\|=1$ and,
in particular, $y\neq 0$.  Moreover, by the continuity of $|\cdot|$ it follows that $|y|=0$.
In conclusion,  $y \in Y \cap K$ and $y\neq 0$. 

Next we claim that $\lambda y$ belongs to $K$ for any $\lambda>0$. Indeed, because of $\|v_n\|\to \infty$, there exists $n_0$ such that $\|v_n\|>\lambda$ for all $n\geq n_0$. By convexity, $\lambda y_{n}\in K$ for all $n\geq n_0$, and by the closedness of $K$, $\lambda y \in K$. Hence,  if $Y\cap K$ is bounded, the existence of $y \in Y \cap K$, $y\neq 0$,  leads immediately to a contradiction.
Otherwise, we obtain from (\ref{semieq1}) that 
\begin{equation} \label{semi2}
0 \leq 
 \langle g, v_{n} \rangle + c \|v_n\|.
\end{equation}
Dividing (\ref{semi2}) by $\|v_{n}\|$  we arrive in the limit at 
$$
0 \leq \langle g,y \rangle + c,
$$
which is a contradiction to 
(iii).
\qed
\end{proof} 

Some comments on the conditions (A2$^s$) are in order now.
For bilateral contact with given 
$\underline{h} \leq \overline{h}$ in $L^\infty(\Gamma_c)$ on the boundary part $\Gamma_c$ of a bounded domain $\Omega$ 
and the constraint set $\tilde{K} = \{ u \in V:~ \underline{h} \leq u|\Gamma_c \leq \overline{h} \}$ in a Sobolev space $V$ on $\Omega$,
simply set $K = \tilde{K} - \tilde {\underline{h}} $, 
where $\tilde {\underline{h}} \in V$ extends $\underline{h} $ to $\Omega$.
Then the existence theorem applies to the bounded set $Y \cap K$. -  
To simplify the interpretation of the conditions involved in the more delicate unbounded case, assume $V$ is a Hilbert space, as in the subsequent section on HVIs.
Condition (A2$^s$) $(ii)$ 
postulates that the applied force $g$ forms an obtuse angle 
with the directions $y$ of escape.
In contrast, condition $(iii)$ demands together with the Cauchy-Schwarz inequality
that for all $y \in Y\cap K \; \mbox{with} \; \|y\|=1. $
$$
-1 \leq  \frac{1}{\|g\|_{*}}~  \langle g, y \rangle < - \frac{c}{\|g\|_{*}} .
$$
This means that the directions $y$ of escape should stay in a given angle range 
with the applied force $g$ and moreover $g$ should be large enough,
$\| g \| > c$. 

\section{General approximation acheme for semicoercive pseudomonotone variational inequalities} \label{Sec:3}
Let $T$ be a directed set  
and  $\{V_t\}_{t\in T}$ be a family of finite-dimensional subspaces of $V$.  While $K$ is contained in $V$, $K_t$ is a nonempty, closed convex subset of $V_t$, not necessarily contained in $K$. Therefore, for the approximation of $K$ by $K_t$, we employ Mosco convergence, see the  hypotheses (H1) - (H2) below.
\begin{description}
\item[(H1)] If  $\{v_{t'}\}_{t'\in T'}$ weakly converges to $v$ in $V$, $v_{t'} \in
  K_{t'} ~(t' \in T') $ for a subnet $\{K_{t'}\}_{t'\in T'}$ of the  net  $\{K_t\}_{t\in T}$,  then   $v\in K$.
\item[(H2)] For any $v\in K$ and any $t \in T$ there exists $v_{t} \in K_{t}$ 
such that  $v_{t} \to v$  in $V$.
\end{description}
Since $0\in K$, by a translation argument, we can simply assume that $0\in K_t$ for all $t\in T$. 

Further, we replace $\varphi$ by  some some approximation $\varphi_t$ satisfying
\begin{description}
\item[(H3)] For any $t \in T$, $\varphi_t$ is pseudomonotone,
$\varphi_t(\cdot, v)$ is upper semicontinuous on the intersection of $K_t$ with any finite dimensional subspace of $V$, and
$$
\varphi_t(u_t,0) \leq c\|u_t\| \quad \forall u_t \in K_t \,.
$$ 
\item[(H4)] For any nets $\{u_t\}$ and $\{v_t\}$ 
 such that  $u_t \in K_t$,   $v_t\in K_t$, 
  $u_t \rightharpoonup u$, and $v_t \to v$ in $V$ it follows that 
$$
\displaystyle {\limsup_{t \in T}}\,~ \varphi_t(u_t,v_t) \leq \varphi(u,v) \,.
$$
\end{description}

Altogether the problem $VI(A,\varphi,g,K)$ is approximated by the problem 
$VI(A,\varphi_t,g, K_t)$:  Find $u_t\in K_t$ such that
\[
\langle A u_t, v_t-u_t \rangle + \varphi_t (u_t,v_t)\geq \langle g, v_t-u_t \rangle \quad \forall v_t \in K_t.
\]
 In some computations it will be necessary to replace also $A$ and $g$ by some approximations $A_t$ and $g_t$, defined for example by a numerical integration procedure which is used in the finite element discretization.
 But this is rather standard; therefore we do not elaborate on this aspect. In view of our applications to hemivariational inequalities, we assume also that $\varphi(u,u)=0$ for all $u\in V$.  
\begin{theorem} \label{th:2}
Under assumptions $(A1^s)$-$(A2^s)$, and hypotheses $(H1)-(H4)$, the family $\{u_t\}$ of solutions to $VI(A,\varphi_t,g,K_t)$ is uniformly bounded in $V$. Moreover, there exists a subnet $\{u_{t'}\}_{t'\in T'}$ of $\{u_t\}$ that converges weakly  in $V$ to  a solution
$u$ of  the problem $VI(A, \varphi,g,K)$ and 
satisfies  $\displaystyle \lim _{t'\in T'} |u_{t'}-u |=0$.
\end{theorem}
\begin{proof} 
The existence and uniform boundedness of the family $\{u_t\}$ can be shown by using the same arguments as those used to prove Theorem \ref{th:1}.  
Further, we can extract a subnet, again denoted by
$\{u_{t}\}$, such that 
$
u_{t} \rightharpoonup u.
$
In view of (H1), $u$ belongs to   $K$. 
Since $v \mapsto \langle Av,v\rangle $ is convex and continuous, hence weakly lower semicontinuous,
\begin{equation} \label{wlsc}
\langle Au,u\rangle \leq  \liminf_{t\in T} \, \langle Au_t,u_t \rangle \,.
\end{equation}
Now, take  an arbitrary   $v \in K$. By (H2),  there exists  a net $\{v_t\}$ such that $v_t \in K_t$ and $v_t \to v$ in $V$. 
By (H4) and (\ref{wlsc}), we get from $VI(A, \varphi_t,g,K_t)$ that for any $v \in K$ 
\begin{eqnarray*}
\langle A u, v-u \rangle + \varphi(u,v) &\geq& 
\limsup_{t\in T} \langle A u_t, v_t-u_t \rangle + \limsup_{t\in T} \varphi_t (u_t, v_t) \\ 
&\geq& \limsup _{t\in T}\,\{ \langle A u_t, v_t-u_t \rangle + \varphi_t (u_t, v_t)\} \\ &\geq& \lim_{t\in T} \, \langle g, v_t-u_t \rangle = \langle g, v-u \rangle
\end{eqnarray*}
and consequently, $u$ is a solution to $VI(A, \varphi,g,K)$. 

Finally, we show the convergence with respect to the seminorm $| \dot |$. For this purpose, by (H2), we find a subnet $\{\bar{u}_t\}$, $\bar{u}_t\in K_t$ such that $\bar{u}_t \to u$.
We start with the relation
\begin{equation} \label{semrelation}
c |\bar{u}_t-u_t|^2 \leq 
\langle A \bar{u}_t, \bar{u}_t-u_t \rangle +
\langle A {u}_t, u_t - \bar{u}_t \rangle .
\end{equation}
The first term goes to zero, since $A\bar{u}_t\to Au$ and $\bar{u}_t-u_t \rightharpoonup 0$ in $V$. 

From the definition of $VI(A, \varphi_t,g,K_t)$ it follows that
\[
\langle A u_t, u_t-\bar{u}_t \rangle \leq \langle g, u_t-\bar{u}_t \rangle 
  + \varphi_t(u_t, \bar{u}_t).
\]
Hence by (H4),
\[
\limsup_{t\in T} \langle A u_t, u_t-\bar{u}_t \rangle 
\leq \limsup _{t\in T} \varphi_t(u_t, \bar{u}_t) \leq \varphi (u,u)=0,
\]
and therefore, (\ref{semrelation}) entails in the limit that
$ \limsup_{t\in T} c |\bar{u}_t-u_t|^2 \leq 0 $,
hence, $ | \bar{u}_t-u_t| \to 0 $.
Further, using the triangle inequality, we get in the limit
$$
0 \leq \lim_{t\in T} |u_{t}-u| \leq \lim_{t\in T} |\bar{u}_{t}-u_{t}|_V + \lim_{t\in T}
|\bar{u}_{t}-u|_V = 0
$$
and the proof  is complete. \qed
\end{proof}
\section{Approximation of a semicoercive hemivariational inequality} \label{Sec:4}
Let $V=H^1(\Omega; \R^d)$ $(d=2,3)$ and $K$ be a nonempty closed convex subset of $V$ which will be specified later. Let $\Omega\subset \R^d$  be a bounded domain with a Lipschitz boundary $\partial \Omega$. 
Decompose $\partial \Omega$ into a Neumann part $\Gamma_N$ and a contact part $\Gamma_c$ with positive measure. Note that here the Dirichlet part is assumed to be empty. 

We prescribe  surface tractions $\mathbf {t}  \in L^2(\Gamma_N; \R^d)$ on $\Gamma_N$ and nonmonotone, generally multivalued boundary conditions on $\Gamma_c$. Moreover, we suppose also that the body is subject to volume force $\mathbf{f} \in L^2(\Omega; \R^d)$. 

Let $a(\cdot, \cdot)$ be the  bilinear form of the linear elasticity, i.e. 
\[
a(u,v)=\int_\Omega C_{ijhk}\varepsilon_{ij}(u)\varepsilon_{hk}(v) \, dx 
=: \langle Au,v \rangle, 
\]
where $C_{ijhk}\in L^\infty$ and $\mathcal{C}=\{C_{ijhk}\}$ is assumed to be uniformly positive definite , and $\varepsilon (u)$ is the symmetric strain tensor defined by
\[
\varepsilon_{ij}=\frac{1}{2}(u_{i,j}+u_{j,i}).
\]
Then, the bilinear form is semicoercive and the space of rigid body motions is 
\[
\mbox{ker}\;a (\cdot , \cdot) =\{u \in H^1(\Omega; \R^2) \, : \, a(u,u)=0 \}\neq \{0\}.
\]
In particular, 
\begin{description}
\item(i) if $\Omega \subset \R^2$ then 
$
\mbox{ker}\;a (\cdot , \cdot) = \{ u(x)=( a_1-bx_2,  a_2+ b x_1), \; a_1, a_2, b\in \R, \;\forall x\in \Omega\};
$ \\
\item(ii) if $\Omega \subset \R^3$ then 
$
\mbox{ker}\;a (\cdot , \cdot) = \{ u(x)=a+b \wedge x, \;  a,  b\in \R^3, \; \forall x\in \Omega\}.
$
\end{description}
We define 
\begin{equation} \label{bifunction}
\varphi(u,v):=\displaystyle
{\int_{\Gamma_c}} f^0(\gamma u(s);\gamma v(s)-\gamma u(s)) \, ds \quad \forall
u, v \in V,
\end{equation}
the linear form 
\[
\langle g, v\rangle := \int _\Omega \mathbf{f} \cdot v \, dx+ \int _{\Gamma_N} \mathbf{t} \cdot  \gamma v\, ds, 
\]
and consider the semicoercive hemivariational inequality: Find $u \in K$ such that
\begin{equation} \label{semipr}
\langle Au,v-u \rangle + \varphi(u,v) \geq \langle g, v-u\rangle \quad \forall v\in K.
\end{equation}
Here, $f^0(\xi;\eta)$ is the generalized Clarke derivative \cite{Clarke} of a locally Lipschitz function $f \, :\, \R^d \to \R$ at $\xi\in \R^d$ in the direction $\eta \in \R^d$, and   $\gamma$ stands  for the trace  operator from  $H^1(\Omega;\R^d)$ into $L^2(\Gamma_c;\R^d)$ with the norm $\gamma_0$. It is worthwhile to note that the genaralized Clarke derivative coinsides with the classical directional derivative only for Clarke regular functions, like convex and maximum-type functions.

Further, we require the following growth condition on the locally Lipschitz superpotential 
$f : \R \to \R $:
\[
\begin{array} {ll}
 (i) & |\eta| \leq c_3 (1+|\xi|) \; \mbox{for all}\; \eta \in \partial f (\xi) \; \mbox{with}\; c_3>0; \\[0.2cm]
 (ii) & \eta (-\xi) \leq c_4|\xi|  \; \mbox{for all}\; \eta \in \partial f (\xi) \; \mbox{with} \; c_4>0;
\end{array}
\]
According to \cite[Lemma 15]{Ovcharova} the bifunction $\varphi : V\times V \to \R$ is well-defined, pseudomonotone, upper-semicontinuous and the condition (\ref{ass1}) is satisfied with 
\begin{equation} \label{constant_c}
c:=c_4 \,\mbox{meas}(\Gamma_c)^{1/2} \,\gamma_0.
\end{equation}
According to Theorem \ref{th:1}, the HVI (\ref{semipr}) has at least one solution provided that 
$
\langle g, v\rangle < - c
$
for all $v\in K \cap \, \mbox{ker} \, a(\cdot\,, \, \cdot)$ with $\|v\|=1$. 

The approximation of the HVI is based first on the smoothing of the nonsmooth functional $\varphi$ defined on the contact boundary and then, on discretizing of the regularized problem by finite elements. For more details we refer to  \cite{Ov_Gwin2014,Ovcharova}.
\subsection{Regularization} 
 Let $S : \R^d \times \R_{++} \to \R$ be a continuously differentiable approximation of $f$ in the sense that
\[
\displaystyle \lim_{z \to x, \varepsilon \to 0^+} S(z, \varepsilon)=f(x) \quad \forall x\in \R^d.
\]
Define
$D J_{\varepsilon}: V \to V^*$  by
$$
\langle D J_{\varepsilon}(u), v \rangle = \displaystyle
{\int_{\Gamma_c}}  \nabla_x S(\gamma u(s), \varepsilon)\cdot \gamma v(s)  \,ds.
$$
The regularized problem reads now as follows: ~find $u_{\varepsilon}\in K$ such that 
\begin{eqnarray} \label{regpr}
\langle Au_{\varepsilon}-g,v -u_{\varepsilon}\rangle +
\langle D J_{\varepsilon}(u_\varepsilon), v-u_\varepsilon \rangle \geq 0
\quad \forall v\in K.
\end{eqnarray}
\subsection{Finite Element Discretization} 
We consider a regular triangulation $\mathcal{T}_h$ of $\Omega$ and define
\[
V_h =\{ v_h\in C (\bar{\Omega}; \R^d) \, : \, v_h|_T \in (\p_1)^d, \; \forall T \in \mathcal{T}_h\}
\]
as a space of all continuous piecewise linear functions. Here, $\p_1$ consists of all polynomials of degree at most one.  In addition,   
we have a family $\{K_h\}$ of nonempty closed convex subsets of $V_h$ that will be specified later on, not necessarily contained in $K$, such that  (H1) and (H2) are satisfied. 
We use  trapezoidal quadrature rule to approximate $\langle D J_\varepsilon (\cdot), \cdot \rangle$ as follows
\begin{eqnarray*} 
\langle D J_\varepsilon(u_{h}), v_{h} \rangle &\approx &
   \frac{1}{2}\sum_i
|P_iP_{i+1}|\,\big[ \nabla_x S(
   \gamma u_{h} (P_i),\varepsilon) \cdot \gamma
v_{h}(P_i)  \nonumber +  \nabla_x S(
   \gamma u_{h} (P_{i+1}), \varepsilon) \cdot \gamma
v_{h}(P_{i+1})  \big ] \\ &=: & \varphi_{\varepsilon,h}(u_h, v_h), 
\end{eqnarray*}
where we have summed over all sides $(P_i,P_{i+1})$ of the triangles of $\mathcal{T}^h$ whose union gives the contact boundary $\Gamma_c$.

The discretization of the regularized problem (\ref{regpr}) reads now:  Find $u_{\varepsilon h}\in K_h$ such that
\begin{equation} \label{discpr}
a(u_{\varepsilon h},v_h-u_{\varepsilon h}) + \varphi_{\varepsilon,h} (u_h,v_h) 
\geq \langle g, v_h-u_{\varepsilon h} \rangle \quad \forall v_h \in K_h.
\end{equation}
According to \cite{Ovcharova}, $\varphi_{\varepsilon, h}$ is pseudomonotone and satisfies the hypotheses (H3) and (H4). Therefore, due to Theorem \ref{th:2} with 
$t=(\varepsilon,h) \in T = \R_{++} \times \R_{++}$, there exists a solution $u_{\varepsilon h}$ to the discrete regularized problem (\ref{discpr}), the family $\{u_{\varepsilon h}\}$ is uniformly bounded and any weak accumulation point of $\{u_{\varepsilon h}\}$ is a solution of the problem (\ref{semipr}). 
\section{Numerical Results} \label{sec:4}
\subsection{Statement of the problem}
As a model example we consider an unilateral contact of an elastic body with a rigid foundation under given forces and a nonmonotone friction law on the contact boundary. 
Let  $\Omega$ be the linear elastic body represented by the unit square $1m\times 1m$   with  modulus of elasticity $E= 2.15\times 10^{11} N/m^2$ and Poisson's ration $\nu =0.29$ (steel). The boundary $\Gamma:=\partial \Omega$ is decomposed into a Neumann part and a contact part. We emphasize that in all our benchmark examples no Direchlet boundary  part is assumed. In particular, on the part $\Gamma_3$ of the boundary  we assume that the horizontal displacement $u_1$ is zero, but the vertical displacement $u_2$ is not fixed, see Fig. \ref{fig:1}. 

The linear Hooke's law is given by
\begin{equation} \label{tau}
\sigma_{ij}(\mathbf{u})= \frac{E\nu}{1-\nu^2} \delta_{ij} \, tr \big (\mathbf{\varepsilon} (\mathbf{u}) \big)+
\frac{E}{1+\nu}\varepsilon_{ij}(\mathbf{u}), \quad i,j =1,2,
\end{equation}
where $\delta_{ij}$  is the  Kronecker symbol  and 
$$
tr \big (\mathbf{\varepsilon} (\mathbf{u}) \big):= \varepsilon_{11}(\mathbf{u})+ \varepsilon_{22}(\mathbf{u}).
$$
\begin{figure} 
    \subfigure[Wall left
    ]{\includegraphics[trim=0cm 18cm 0cm 2cm, clip,width=0.55\textwidth]{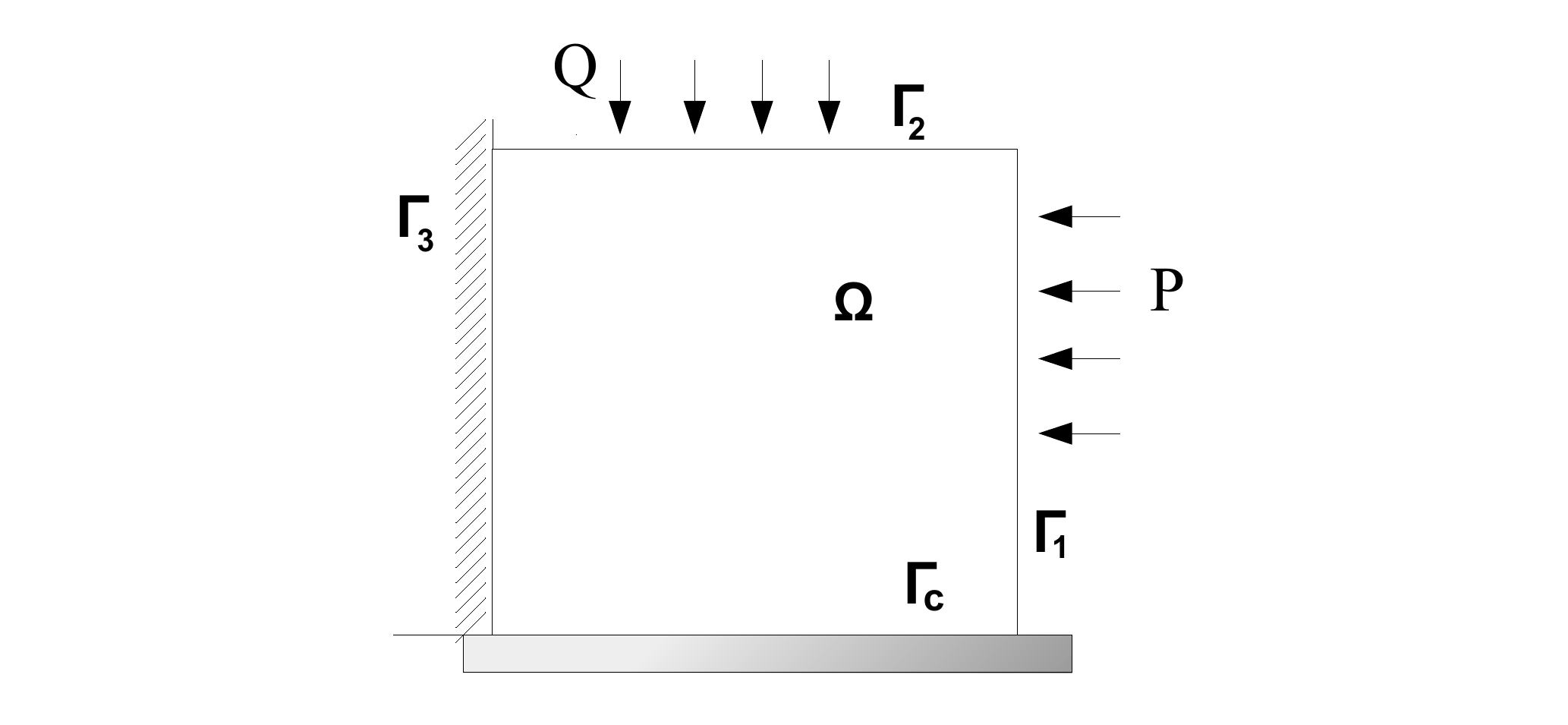}}
    \subfigure[Wall right 
    ]{\includegraphics[trim=0cm 18cm 0cm 2cm, clip,width=0.55\textwidth]{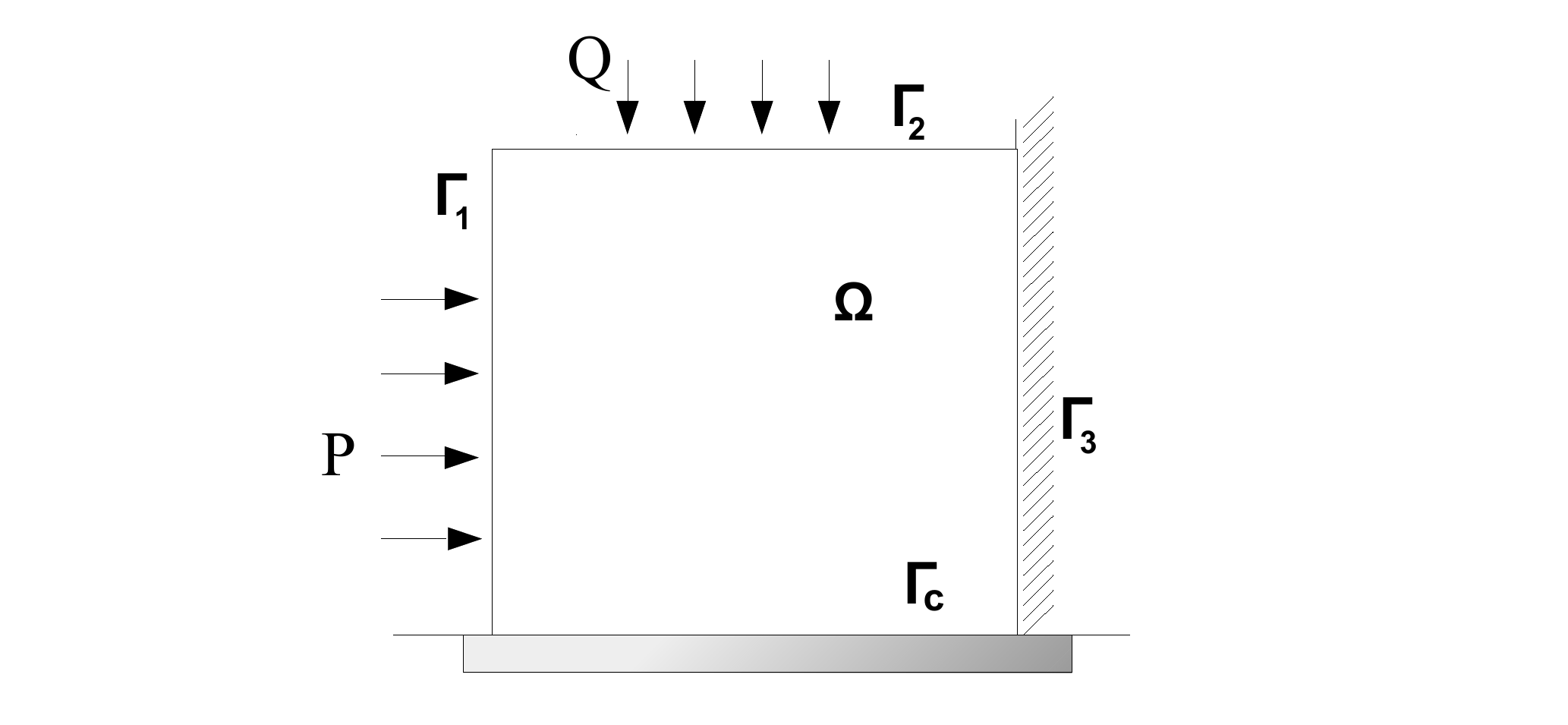}}
\caption{\small 2D benchmark examples with force distribution and boundary decomposition. In both cases $u_1=0$ on ~$\Gamma_3$, $u_2$- 	arbitrary on $\Gamma_3$} \label{fig:1}
\end{figure} 
The body is loaded with horizontal forces $\mathbf{F}_1$ on $\Gamma_{1}$ and vertical forces $\mathbf{F}_2$ on $\Gamma_{2}$. The
volume forces are neglected. In our experiments we have used the data:
\[
\begin{array}{lll}
\mathbf{F}_1=(\pm P,0) & \mbox{with} & P=1 \times 10^6 N/m^2\\[0.1cm]
\mathbf{F}_2=(0,-Q)  & \mbox{with} & Q=1 \times 10^6 N/m^2 .
\end{array}
\]
Further, let ${\mathbf n}$ be the unit outer normal vector on the boundary $\partial \Omega$. The stress vector on the
surface   is decomposed into the normal, respectively, the tangential stress:
\[
 \sigma_n:=\sigma(u)\mathbf{n}\cdot {\mathbf n}, \quad \sigma_t:=\sigma(u)\mathbf{n}-\sigma_n {\mathbf n}.
\]
In addition, we assume that
$$
\left\{
\begin{array}{cl}
    u_2(s)\geq 0 \quad &  
                s\in \Gamma _c\\[0.1cm]
   -\sigma_t(s) \in \partial j( u_1(s)) \quad &  \mbox{for a.a.} \; s\in \Gamma _c .
\end{array}
\right.
$$
The assumed nonmonotone multivalued law $\partial j $ holding in the tangential (horizontal)  direction is depicted in Fig.  \ref{fig:2} 
with parameters $\delta=9.0\times 10^{-6} m$, $\gamma_1=1.0\times 10^3 N/m^2$ and $\gamma_2=0.5\times 10^3 N/m^2. $  Notice that here $j$ is  a minimum of one convex quadratic and  one linear function, i.e.  
\[
j(x)=\min \{\frac{\gamma_1}{2\delta} x^2, \gamma_2 x\}.
\]
\begin{figure}[t]
\begin{center}
{\includegraphics[trim=5cm 12cm 0cm 9cm, clip, width=0.7\textwidth]
{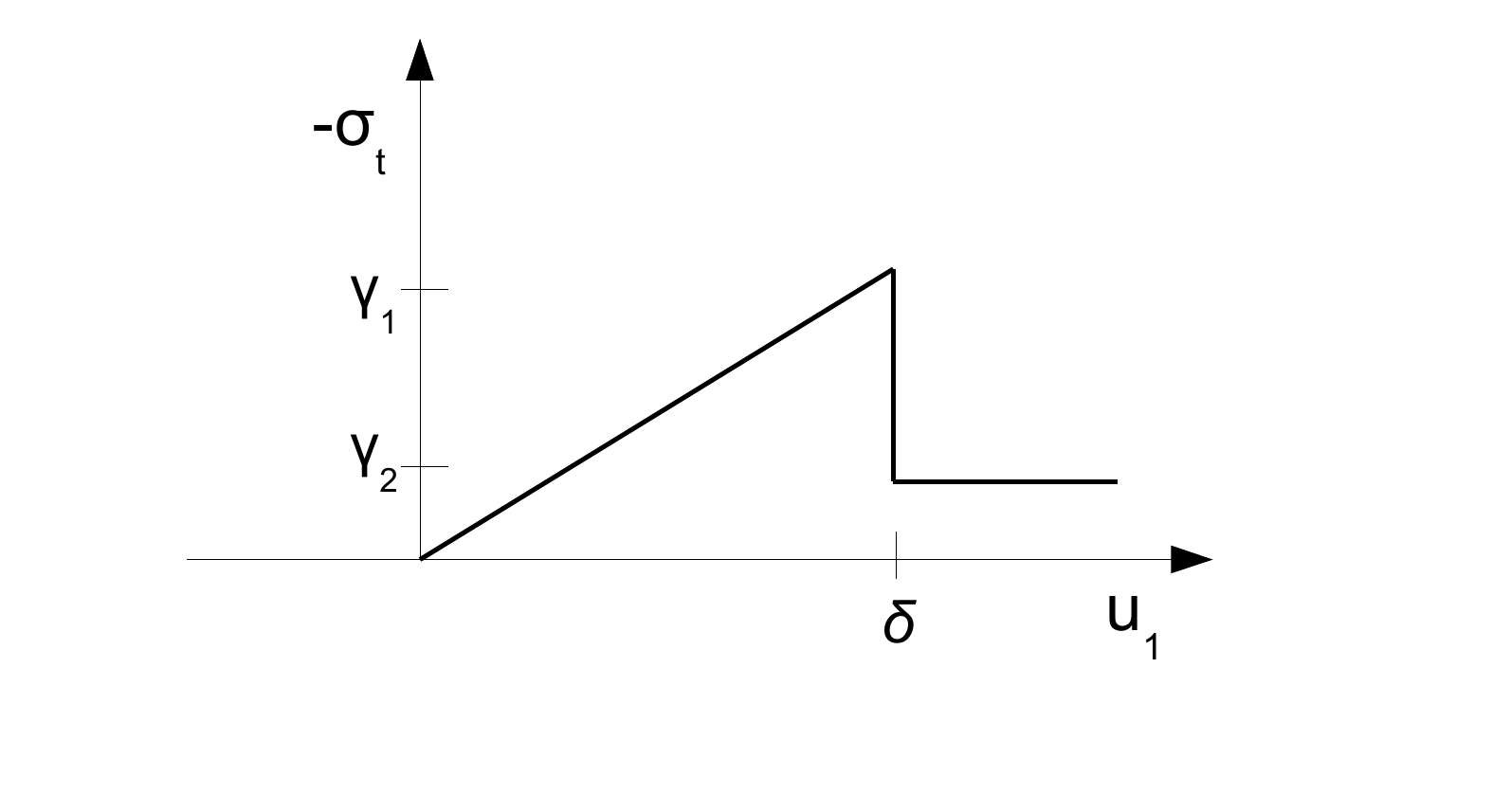}}
\end{center}
\vspace{-0.5cm}
\caption{\small A nonmonotone friction law}\label{fig:2}
\end{figure}
Let
$$
V=\{ \mathbf{v} \in H^1(\Omega; \R^2)\,: \, v_1=0  \, \; \mbox{on} \, \;  \Gamma_{3}\}
$$
and 
$$
K=\{\mathbf{v} \in V \, : \, v_2 \geq 0 \;\mbox{on} \; \, \Gamma_{c}\}
$$
be the convex set of all admissible displacements. The weak formulation of this  contact problem leads to
the following hemivariational inequality: ~Find  
 $\mathbf{u}\in V$ such that for all $\mathbf{v} \in V$
\begin{equation} \label{mech1}
a(\mathbf{u},\mathbf{v}-\mathbf{u}) + \displaystyle \int_{\Gamma_c}j^0(u_1(s);v_1(s)-u_1(s)) \, d s \geq \langle \mathbf{g}, \mathbf {v}- \mathbf {u}
\rangle. 
\end{equation}
Here, $a(\mathbf{u},\mathbf{v})$ is the energy bilinear form of linear elasticity 
\[
a(\mathbf{u},\mathbf{v})=\displaystyle \int_{\Omega}
\sigma_{ij}(\mathbf{u})\varepsilon_{ij}(\mathbf{u})\, dx \quad \mathbf{u}, \mathbf{v} \in V
\]
with $\mathbf{\sigma}$, $\mathbf{\varepsilon}$ related by means of (\ref{tau})
and the linear form $\langle \mathbf{g}, \cdot \rangle$ defined by
\[
\langle \mathbf{g}, \mathbf{v} \rangle = \pm P \int _{\Gamma_{1}}v_1 \, ds -Q \int _{\Gamma_{2}}v_2 \, ds. 
\]
From Theorem \ref{th:1}, 
 we obtain the existence of at least one solution provided that
\[
\pm P \int _{\Gamma_{1}} a_1-b x_2(s) \, ds -Q \int _{\Gamma_{2}}a_2+bx_1(s) \, ds < - \gamma_2 \, \mbox{meas}\, (\Gamma_c) ^{1/2} 
\]
for all $a_1, a_2, b$ $\in \R$ satisfying $\int_{\Gamma_3}a_1-bx_2(s)\, ds=0$ and $\int_{\Gamma_c}a_2+bx_1(s)\, ds \geq 0$. 

Note that now we use the decomposition of $\mathbf{u}$ into normal and tangental parts and therefore, $\gamma_0=1$ in (\ref{constant_c}).


\subsection{Regularization and discretization}
We  solve this problem numerically following the method presented in Section \ref{Sec:4} by first regularizing the  hemivariational inequality (\ref{mech1}) and then discretizing  the regularized problem by the finite element method. 
This procedure leads to a  smooth optimization problem that can  be finally solved by using global minimization algorithms like trust region methods. 

For this purpose we fix $\varepsilon$ and use $S :  \R \times \R_{++}\to \R$ defined by
\[
S(x, \varepsilon): = 
\left \{ \begin{array} {ll} g_1(x) \quad & \mbox{if} \, \;
    (i) \; \mbox{holds} 
     \\[0.1cm]
\frac{1}{2\varepsilon} [g(x) -g_1(x)]^2  + \frac{1}{2} (g(x) +
g_1(x)) + \frac{\varepsilon}{8} \quad & \mbox{if} \, \; (ii)  \; \mbox{holds}  
\\[0.1cm]
g_2(x) \quad & \mbox{if} \, \; (iii) \;  \mbox{holds},
\end{array}
\right.
\]
to approximate the maximum function $-j(x)=\max\{-\frac{\gamma_1}{2\delta} x^2, -\gamma_2 x\}$. The cases 
$(i)$,  $(ii)$,  $(iii)$ are defined below,  respectively, by
\begin {description}
\item (i) $g_2(x) -g_1(x) \leq - \frac{\varepsilon}{2} $ ~\\[0.1cm]
\item (ii) $- \frac{\varepsilon}{2} \leq
g_2(x) -g_1(x) \leq \frac{\varepsilon}{2}$  ~\\[0.1cm]
\item (iii) $g_2(x) -g_1(x) \geq \frac{\varepsilon}{2}$.
\end{description}
Let $\{\mathcal T_h\}$ be a regular triangulation of $\Omega$ and 
$\{x_i\}$ be the set of all vertices of the triangles of $\{ \mathcal {T}_h\}$. 
We use continuous piecewise linear functions to approximate the displacements. 
Thus,  $V$ and $K$ are approximated, respectively,  by
$$
V_h=\{v_h \in C(\overline{\Omega};\R^2)\, : \, {v_h}_{|_{T}} \in (\p_1)^2,
\; \; \forall T \in \mathcal T_h, \, \; v_{h1}(x_i) =0  \; \; \forall x_i \in \overline{\Gamma}_3 \}
$$
$$
K_h=\{v_h \in V_h \, : \,  v_{h2}(x_i) \geq 0  \; \; \forall x_i \in \overline{\Gamma}_c \}.
$$
The approximation of (\ref{mech1}) now reads as follows: ~Find $u_h\in V_h$ such that 
$$
a(u_h,v_h-u_h) + \langle - D J_h(u_h),v_h-u_h\rangle \geq P \int
_{\Gamma_F^1}(v_{h1}-u_{h1})\,d x_2 \quad  \forall v_h\in K_h,
$$
where 
\begin{equation} \label{reg_dis}
\langle D J_h(u_h),v_h\rangle  =  \frac{1}{2} \sum |P_iP_{i+1}| \Big [ 
\frac{\partial S}{\partial x} (u_{h1}(P_i),\varepsilon) v_{h1}(P_i) + \frac{\partial S}{\partial x} (u_{h1}(P_{i+1}),\varepsilon) v_{h1}(P_{i+1})
\Big ]. 
\end{equation}
The discretized regularized problem (\ref{reg_dis})  is put to work using the following steps. First, we use a condensation technique based on a Schur complement to reduce the total number of unknowns in (\ref{reg_dis}) and pass to a finite-dimensional variational inequality problem formulated only in terms of the contact displacements. The obtained problem is re-written as a mixed complementarity problem, which is further reformulated as a system of nonlinear equations by means of the Fischer-Burmeister function $f(a,b)=\sqrt{a^2+b^2}- (a+b)$. Finally, we apply an appropriate merit function and  obtain an equivalent smooth, unconstrained minimization problem, which is numerically solved by using {\it lsqnonlin} - MATLAB function based on trust region method. The maximal number of iteration in {\it lsqnonlin} has been fixed to $100$. The
regularization parameter $\varepsilon$ is set to $0.1$.
The tangential component $u_1$ along $\Gamma_c$ for the two models (see Fig. \ref{fig:1}) and four different mesh sizes $h=1/4, 1/8, 1/16$ and  $1/32$ in ~[m] is captured in Fig. \ref{fig:4} (a), respectively,  Fig. \ref{fig:5} (a). The computed tangential stress $-\sigma_t$ along $\Gamma_c$ is shown in Fig. \ref{fig:4} (b), respectively, Fig. \ref{fig:5} (b). Fig. \ref{fig:3} illustrates the computed complete  displacement field on the whole boundary $\Gamma$ for same mesh sizes.  
\begin{figure}[b]
    \subfigure[Reference and deformed configurations in cm corresponding to Fig. \ref{fig:1} (a)]{\includegraphics[trim = 0.5cm 8cm 0.5cm 9cm,clip,width=1.0\textwidth]{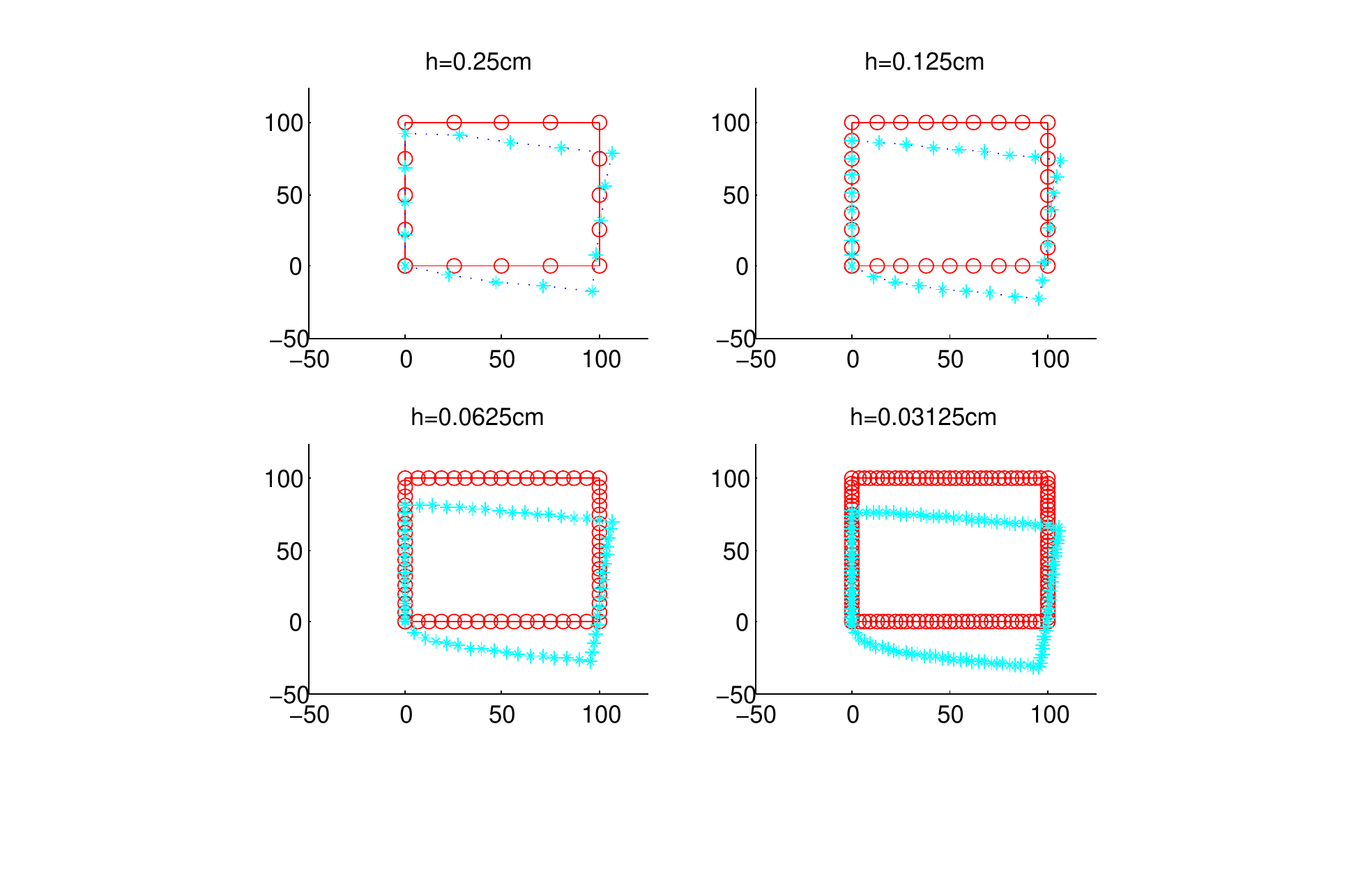}} 
    \subfigure[Reference and deformed configurations in cm corresponding to Fig. \ref{fig:1} (b)]{\includegraphics[trim = 0.5cm 8cm 0.5cm 9cm,clip,width=1.0\textwidth]{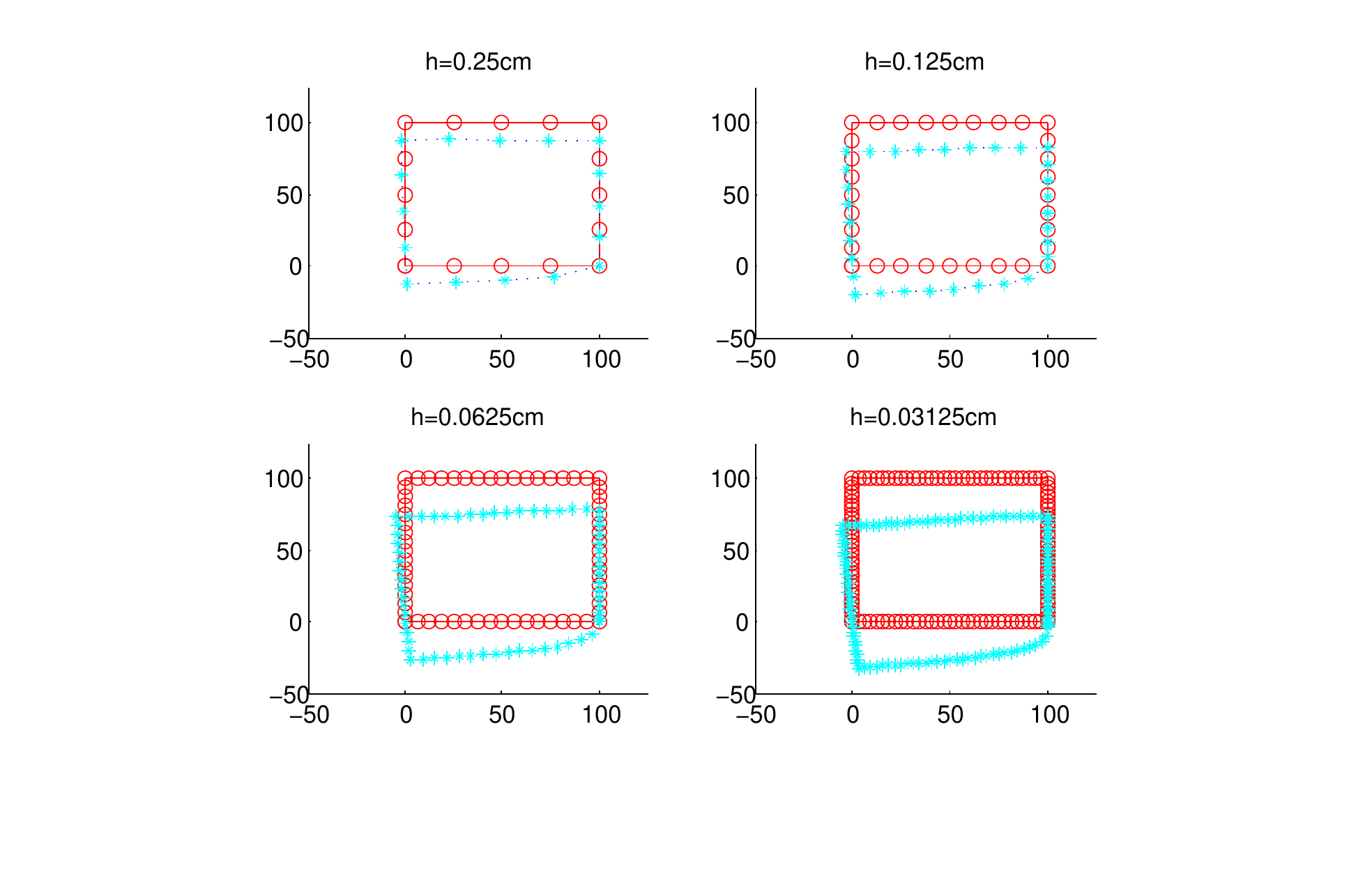}}
\caption{The complete displacement field on the whole boundary $\Gamma$} \label{fig:3}
\end{figure}
\begin{figure}[h!] 
\centering
\subfigure[ $u_1$ on $\Gamma_c$]
{\includegraphics[trim = 1cm 8cm 0.5cm 9cm,clip,width=0.49\textwidth]{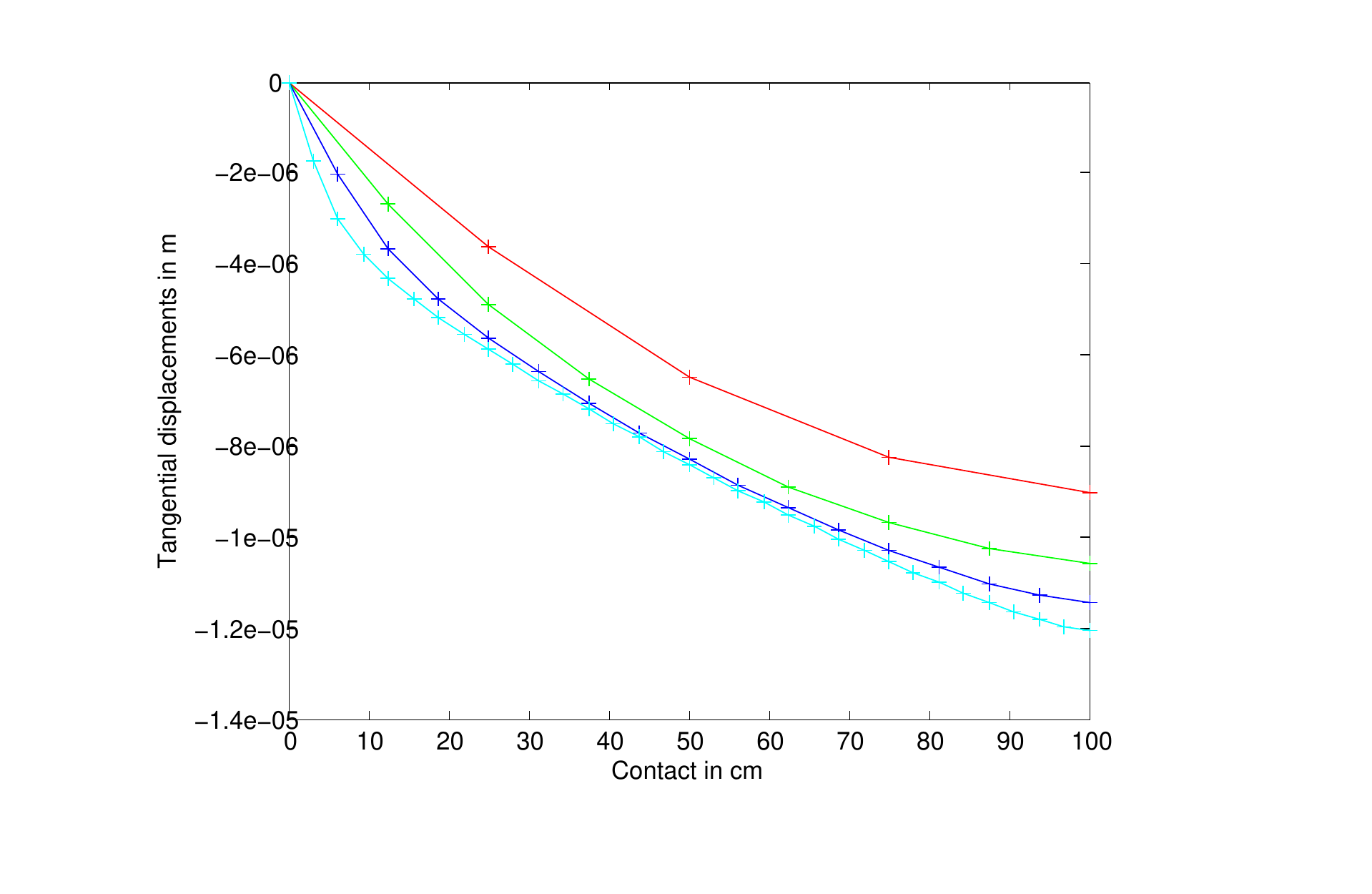}} 
\subfigure[$-\sigma_t$ on $\Gamma_c$]
{\includegraphics[trim = 1cm 8cm 0.5cm 9cm,clip,width=0.49\textwidth]{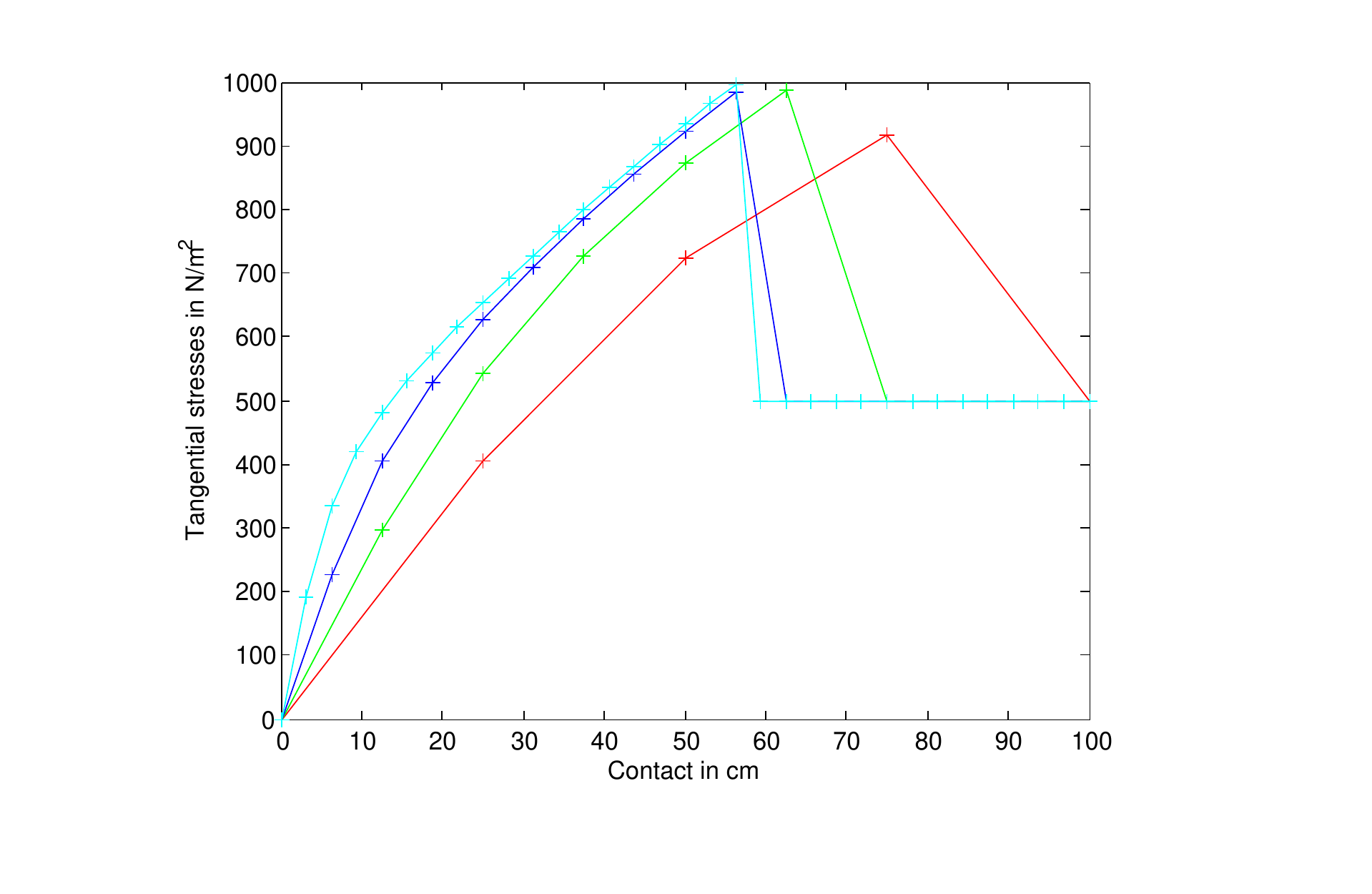}} 
\caption{
{\it Wall left:} Left image shows the tangential component $u_1$  on $\Gamma_c$ for $4$ discretization parameters $h=1/4$ ({\it red}), $h=1/8$ ({\it green}), $h=1/16$ ~({\it dark blue}), $h=1/32$ ({\it light blue}) in [m]. The right image shows the distribution of the tangential stress $-\sigma_t$ along $\Gamma_c$ for the same $4$ scenarios}
\label{fig:4}
\end{figure}

\begin{figure}[h!]
\centering
\subfigure[ $u_1$ on $\Gamma_c$]
{\includegraphics[trim = 1cm 8cm 0.5cm 9cm,clip,width=0.49\textwidth]{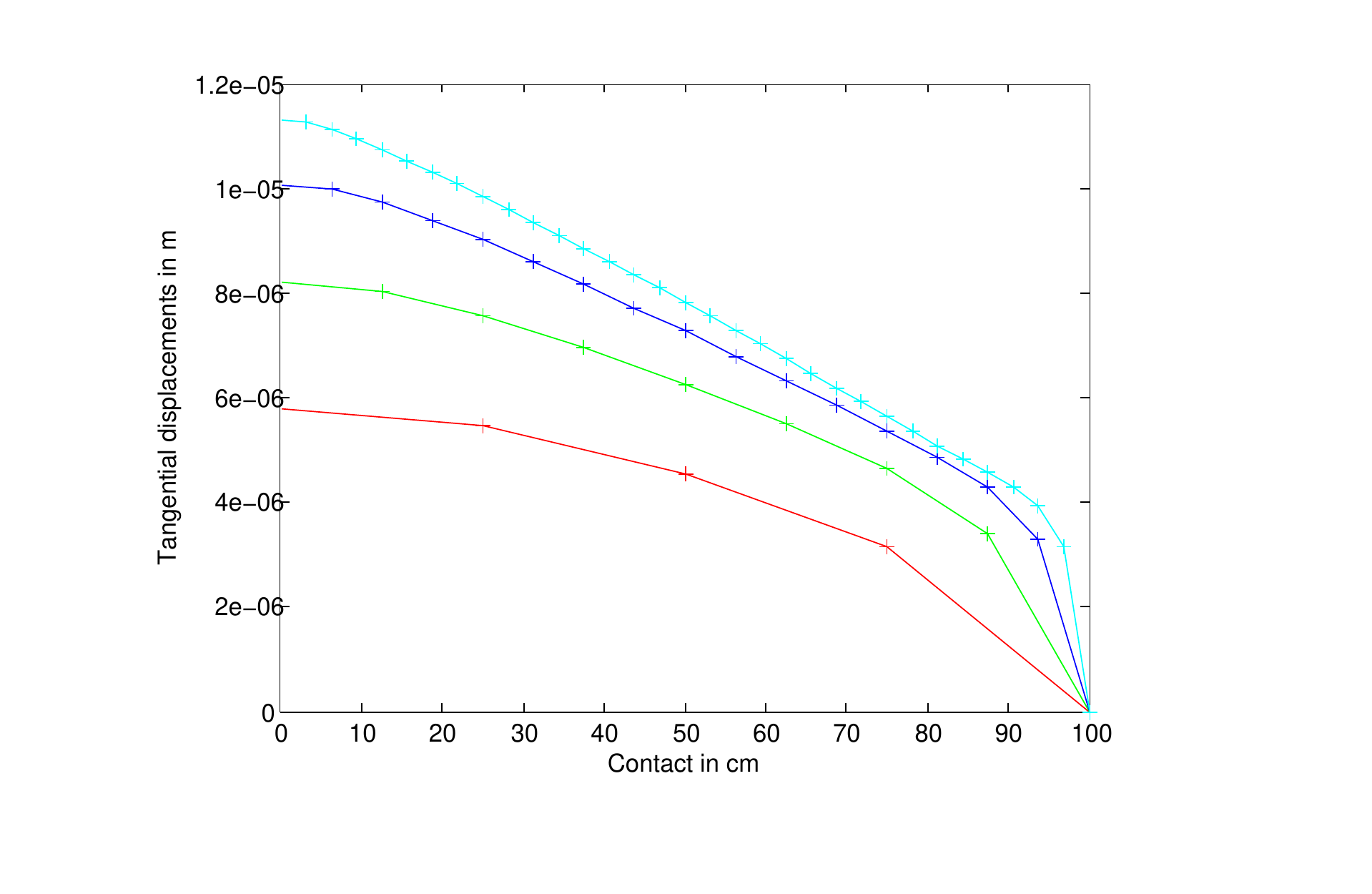}}
\subfigure[$-\sigma_t$ on $\Gamma_c$]
{\includegraphics[trim = 1cm 8cm 0.5cm 9cm,clip,width=0.49\textwidth]{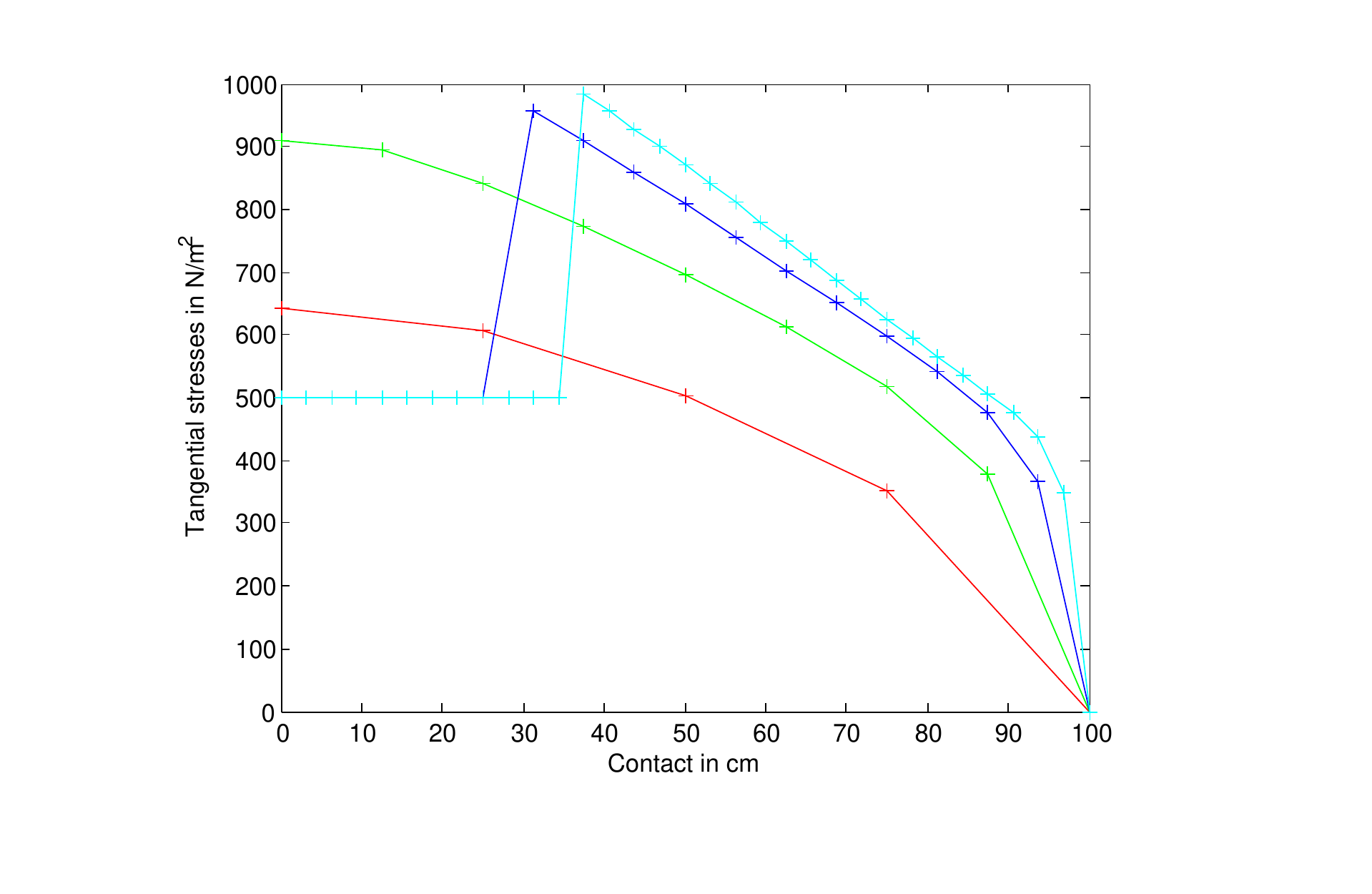}}
\caption{
{\it Wall right:}  Left image shows the tangential component $u_1$  on $\Gamma_c$ for $4$ discretization parameters $h=1/4$ ({\it red}), $h=1/8$ ({\it green}), $h=1/16$ ~({\it dark blue}), $h=1/32$ ({\it light blue}) in [m]. The right image shows the distribution of the tangential stress $-\sigma_t$ along $\Gamma_c$ for the same $4$ scenarios}
\label{fig:5}
 \end{figure} 
 

\end{document}